\documentclass[12pt]{amsart}
\usepackage{amsmath, amssymb, amsfonts, amsthm, enumerate, graphicx }

\theoremstyle{plain}

\newtheorem{theorem}{Theorem}[section]
\newtheorem{theorem*}{Theorem}
\newtheorem{lemma}[theorem]{Lemma}

\newtheorem{corollary}[theorem]{Corollary}
\newtheorem{prop}[theorem]{Proposition}
\newtheorem{const}[theorem]{Construction}

\author{James Phillips}

\title{One-Point Covers of Elliptic Curves and Good Reduction}

\begin{document}

\maketitle

\begin{abstract}

Raynaud gave a criterion for a branched $G$-cover of curves defined over a mixed-characteristic discretely valued field $K$ with residue characteristic $p$ to have good reduction in the case of either a three-point cover of $\mathbb{P}^1$ or a one-point cover of an elliptic curve. Specifically, such a cover has potentially good reduction whenever $G$ has a Sylow $p$-subgroup of order $p$ and the absolute ramification index of $K$ is less than the number of conjugacy classes of order $p$ in $G$. In the case of an elliptic curve, we generalize this to the case in which $G$ has an arbitrarily large cyclic Sylow $p$-subgroup.

\end{abstract}

\section{Introduction}

Let $f:Y \to X$ be a $G$-Galois cover branched at $r$ points. For the case of $X = \mathbb{P}^1$ and $r=3$, in \cite{raynaud99}, Raynaud gives a criterion for such covers to have good reduction to characteristic $p$ over mixed-characteristic discretely valued fields. Specifically, $f$ has good reduction whenever $p$ strictly divides the order of the Galois group, $G$, of this cover and the absolute ramification index of the field is less than the number of conjugacy classes of elements of order $p$ in $G$. Motivating this was the surjection due to Grothendieck (\cite{sga1}) of tame fundamental groups:

\[pr: \pi_1(\mathbb{P}^1_K \setminus \{0,1,\infty\})^{p-tame} \to \pi_1(\mathbb{P}^1_k \setminus \{0,1,\infty\})^{tame}. \]

Understanding the kernel of this map amounts to determining the $p$-tame three-point covers having good reduction.

In the same paper, Raynaud suggested that one may obtain a similar result for both three-point covers and \textit{one-point covers} (that is to say, covers of elliptic curves branched over one point) whose Galois groups have arbitrarily large cyclic Sylow $p$-subgroups. Indeed, Raynaud's own result holds also in the one-point cover case, and the two cases often have analogous properties: for instance, both have the same \'etale fundamental group. Moreover, in \cite{obus17}, Obus obtained this very result in the three-point case. We complete the analogy in the following:

\begin{theorem*}

Let $G$ be a finite group with cyclic Sylow $p$-subgroup. Let $k$ be an algebraically closed field of positive characteristic $p$, $K_0 = \text{Frac}(W(k))$, and $K$ be a finite extension of $K_0$ such that the absolute ramification index of $K$ is less than the number of conjugacy classes of order $p$ in $G$. Let $f: Y \to E$ be a one-point $G$-cover defined over $K$, where $E$ is an elliptic curve having good reduction. Then $f$ has potentially good reduction.

\end{theorem*}

Raynaud's proof relied on the study of the stable reduction of such covers. One obtains the stable model after at most a finite base change of the ground field. The Galois group of this extension is called the \textit{monodromy group} of the cover and its $p$-part is called the \textit{wild monodromy}. Raynuad showed that this wild monodromy of such an $f$ is trivial, which sufficed to conclude potentially good reduction.

An analogous result holds for covers having cyclic Sylow $p$-subgroups, which is also due to Obus in \cite{obus12}. He instead, however, concluded good reduction from the torsor structure of the cover at specific components. We follow a similar route. We must account for the fact that a cover of an elliptic curve need not ramify, so that we have a ``borderline" case which does not arise in the case of $\mathbb{P}^1$. Such covers, however, are necessarily isogenies of elliptic curves, so this does not present much new difficulty.

In \S2, we recall some facts regarding stable models of covers and their ramificaiton. In particular, we have Obus' (\cite{obus12}) \textit{vanishing cycles formula}, which places restrictions on the ramification based on the genus of the base curve and the number of branch points. In \S3, we review the notion of \textit{multiplicative reduction}. One can view reduction to characteristic $p$ as ``losing" information and multiplicative reduction as being a way of ``recovering" some of this lost information. We also outline the construction of the \textit{auxiliary cover}, a cover obtained from the original cover but whose Galois group will be of the form $\mathbb{Z}/p^s \rtimes \mathbb{Z}/m$, where $m$ is prime to $p$. We prove an analogue of a result of Wewers (\cite{wewers03}) on Jacobians of tame cyclic covers of elliptic curves to detect multiplicative reduction in this auxiliary cover. Finally, in \S4, we adapt a result of Obus (\cite{obus17}) to show that a one-point cover defined over a sufficiently small field (for instance, one described in Theorem 1) having bad reduction cannot have multiplicative reduction over certain components, allowing us to conclude good reduction.

\subsection{Notation}

Throughout this paper, a \textit{one-point cover} will refer to a cover of an elliptic curve branched over one point; we typically assume this point to be the origin. A \textit{$G$-cover} of $K$-curves refers to a finite map of curves $f: Y \to X$, where both $X$ and $Y$ are smooth, projective, and geometrically integral $K$-curves and the extension $K(Y)/K(X)$ is Galois with group $G$. When $G$ has a cyclic Sylow $p$-subgroup, $P$, we denote by $m_G$ the value $|N_G(P)|/Z_G(P)|$.

Given an arbitrary scheme morphism $f:Y \to X$ and a finite group $G$ having a subgroup $H$ with $H \subseteq \text{Aut}(Y/X)$, $\text{Ind}_H^G f: \text{Ind}_H^G Y \to X$ is the map obtained by taking the disjoint union of $[G:H]$ copies of $Y$ and applying $f$ to each copy.

\section*{Acknowledgements}

The author would like to thank Andrew Obus for his invaluable comments, discussions, and patience throughout this project.

\section{Stable reduction}

\subsection{The stable model}

Let $R$ be a complete discrete valuation ring with characteristic 0 fraction field $K$ and and positive characteristic $p$ residue field $k$. Let $X$ be a curve defined over $K$ having good reduction; we denote a fixed good model of $X$ by $X_R$. Let $f: Y \to X$ be a $G$-Galois cover defined over $K$ with prime-to-$p$ branching indices, where $G$ is a finite group with cyclic Sylow $p$-subgroup and $Y$ is a curve of genus at least 2; we suppose that the branch points of $f$ are defined over $K$ and specialize to distinct points of $X_k = X_R \times_R k$.

We know from Deligne and Mumford (\cite{dm}, Corollary 2.7) that there is a finite extension of $K$ for which we have a \textit{stable model} $Y^{st}$ of $Y$ over the valuation ring of this extension; that is to say, the special fiber $\bar{Y}$ has only ordinary double points as singularities and any irreducible component of $\bar{Y}$ having genus 0 contains at least 3 marked points n our situation, (that is, ramification points or points of intersection with the rest of $\bar{Y}$).

We combine this with the work of Raynaud in \cite{raynaud99} and Liu in \cite{liu06} to obtain the \textit{stable model} of $f$, $f^{st}: Y^{st} \to X^{st}$, where $Y^{st}$ is the stable model of $Y$, as above, $X^{st} = Y^{st} / G$, and the ramification points on the generic fiber of $f$ specialize to distinct smooth points of $\bar{Y}$. The cover $f^{st}$ is defined over the valuation ring $R^{st}$ of a minimal finite extension $K^{st}$ of $K$.

We call the special fiber $\bar{f}: \bar{Y} \to \bar{X}$ the \textit{stable reduction} of $f$. When $\bar{Y}$ is smooth, we say that $f$ has \textit{potentially good reduction}; elsewise, we say that $f$ has \textit{bad reduction}.

We can view $X^{st}$ as a blow-up of $X_R \times_R R^{st}$. We call the strict transform of the special fiber under this blow-up the \textit{original component}, which we denote $\bar{X}_0$.

\subsection{Ramification on the special fiber}

The action of $G$ on $Y^{st}$ induces an action on $\bar{Y}$. By \cite{raynaud99}, Proposition 2.4.11, the inertia groups of this action at the generic points of $\bar{Y}$ are $p$-groups. If $\bar{V}$ is an irreducible component of $\bar{Y}$, we will denote by $I_{\bar{V}}$ and $D_{\bar{V}}$ the inertia and decomposition groups, respectively, at the generic point of $\bar{V}$.

In $\bar{X}$, the inertia groups at the generic point of an irreducible component $\bar{U}$ are conjugate $p$-groups; if they have order $p^i$, we call $\bar{U}$ a \textit{$p^i$ component}. We say that $\bar{U}$ is \textit{\'etale} when $i=0$ and \textit{inseparable} otherwise. This comes from the fact that, since $\bar{Y}$ is reduced, the inertia arises from an inseparable extension of residue fields at this generic point. By \cite{obus12}, Corollary 2.11, if $\bar{U}$ and $\bar{U}^{\prime}$ are two intersecting components of $\bar{X}$, then either $I_{\bar{U}} \subseteq I_{\bar{U}^{\prime}}$ or $I_{\bar{U}^{\prime}} \subseteq I_{\bar{U}}$ (we stress that this relies crucially on the Sylow $p$-subgroup being cyclic).

We call an irreducible component $\bar{U}$ a \textit{tail} if it intersects the rest of $\bar{X}$ at exactly one point; otherwise, it is an \textit{interior component}. If, in addition, a tail contains the specialization of a branch point, we call this tail \textit{primitive}; otherwise, we say that it is \textit{new}. This follows the convention established, for example, in \cite{raynaud99}. We index the new \'etale tails and primitive \'etale tails by the sets $B_{new}$ and $B_{prim}$, respectively.

Now let $x$ be a point of intersection of two components $\bar{U}$ and $\bar{U}^{\prime}$ with $I_{\bar{U}^{\prime}} \subseteq I_{\bar{U}}$ and let $y$ be a point lying above $x$, the intersection point of two components $\bar{V}$ and $\bar{V}^{\prime}$ lying above $\bar{U}^{\prime}$ and $\bar{U}$, respectively. Following \cite{obus12}, we call the conductor of higher ramification (as in \cite{serre}, for instance) of the extension of complete discrete valuation rings $\hat{\mathcal{O}}_{\bar{U},x}\hookrightarrow \hat{\mathcal{O}}_{\bar{V},y}$ the \textit{effective ramification invariant} at $x$, denoted $\sigma_x$. If $x$ is on a tail $\bar{X}_b$ we will often simply write $\sigma_b$ for $\sigma_x$.

We recall the following, which places some bounds on these invariants; it is originally found in \cite{obus12} in the form of Lemmas 2.20 and 4.2:

\begin{lemma}
\label{bounds}

The effective ramification invariants $\sigma_b$ are positive and lie in $\frac{1}{m_G} \mathbb{Z}$. If $\sigma_b$ is a new tail, we have $\sigma_b \geq 1+\frac{1}{m_G}$

\end{lemma}

In the case in which $G = \mathbb{Z}/p^s \rtimes \mathbb{Z}/m$ and $f: Y \to E$ is a $G$-cover branched at the points $P_1, \ldots, P_r$, $r > 1$, where $E$ is an elliptic curve having good reduction over $R$, we can be more precise as to the fractional parts of the $\sigma_b$, to be denoted $\langle \sigma_b \rangle$. We can decompose $f$ into an \'etale $\mathbb{Z}/p^n$-cover $Y \to Z$ and a $\mathbb{Z}/m$-cover $g: Z \to E$ given, birationally, by $z^m = f_r \cdot f_u$, where $f_r$ and $f_u$ are rational functions on $E$ corresponding to, respectively, a divisor in the form $\sum_{i=1}^r a_i P_i$ and an $m$-divisible divisor on $E$. The following is a generalization of \cite{wewers03}, Proposition 1.8, originally found in \cite{obus17}, Proposition 3.6. It was originally proven in the case in which the base curve is $\mathbb{P}^1$, but can be proven identically in the case in which the base curve is an elliptic curve $E$.

\begin{prop}
\label{fractionpart}

Let $f: Y \to E$ be a $G = \mathbb{Z}/p^s \rtimes \mathbb{Z}/m$-cover branched in $r \geq 1$ points and let $f^{ss}$ be a fixed semistable model for $f$; let $\bar{X}_b$ be an \'etale tail of $\bar{E}^{ss}$ containing the specialization of a unique branch point $x_i$ and intersecting an inseparable component of $\bar{E}^{ss}$. Then $\langle \sigma_b \rangle = \frac{a_i}{m}$, where the former denotes the fractional part of $\sigma_b$.

\end{prop}

In \cite{raynaud99}, Corollary 3.4.4, we find the \textit{vanishing cycles formula}, which relates the invariants $\sigma_b$ and the genus of $X$ in the case in which $G$ has a Sylow $p$-subgroup of the form $\mathbb{Z}/p$. This formula is generalized to the case in which $G$ has a \textit{cyclic} Sylow $p$-subgroup in \cite{obus12}, Theorem 3.14; we recall this version below:

\begin{theorem} 
\label{vanishingcycles}

Let $f: Y \to X$ be a $G$-Galois cover branched at $r$ points having bad reduction, where $G$ has cyclic Sylow $p$-subgroup; let $g$ be the genus of $X$; let $B_{et}$ be an indexing set for the \'etale tails, each having effective ramification invariant $\sigma_b$. Then we have

\[2g - 2 + r = \sum_{b \in B_{et}} (\sigma_b - 1)\]

\end{theorem}

In particular, when $f$ is an elliptic curve, we have 

\[r = \sum_{b \in B_{new}} (\sigma_b - 1) + \sum_{b \in B_{prim}} \sigma_b\]

\section{Reduction types}

We maintain the assumptions of \S 2; namely, that $f: Y \to E$ is a $G$-cover branched at $r \geq 1$ points with prime-to-$p$ branching indices, $G$ having a cyclic Sylow $p$-subgroup, and $Y$ having genus at least 2. Now we will consider the torsor structure of the reduction of one-point covers $f: Y \to X = E$. These will eventually play a key part in the proof of the main theorem.

\subsection{Multiplicative reduction}

Supposing that $f$ has bad reduction, we fix a semistable model $f^{ss}: Y^{ss} \to E^{ss}$. Let $\bar{U}$ be an inseparable component of $\bar{E}^{ss}$ and let $\bar{V}$ be a component of $\bar{Y}^{ss}$ lying above $\bar{W}$ having decomposition and inertia group $D_{\bar{V}}$ and $I_{\bar{V}}$, respectively. The group $D_{\bar{V}}$ has a normal subgroup of order $p$, since $I_{\bar{V}}$ is normal in $D_{\bar{V}}$ and $I_{\bar{V}}$ is cyclic. Then \cite{obus12}, Corollary 2.4, implies that $D_{\bar{V}}$ has a maximal normal prime-to-$p$ subgroup with $D_{\bar{V}}/N \cong \mathbb{Z}/p^j \rtimes \mathbb{Z}/m_{D_{\bar{V}}}$ for some $j>0$; we denote by $P$ the Sylow $p$-subgroup of $D_{\bar{V}}/N$.

Let $\eta$ be the generic point of $\bar{V}/N$. The group $P$ acts on $\hat{\mathcal{O}}_{Y_R/N, \eta}$; we denote this ring by $B$ and its subring fixed by $P$ by $A$. The action of $P$ on the residue field of $B$ is trivial, so that $B$ is a totally ramified extension of $A$. We say that $f$ has \textit{multiplicative reduction} when $B/A$ has the structure of a $\mu_{p^j}$-torsor.

Following \cite{obus17}, we have an equivalent formulation which will be of use to us. We say that $B/A$, as above, is \textit{of $\mu_{p^j}$-type} if $\text{Frac}(B)$ is a Kummer extension of $\text{Frac}(A)$ given by the adjunction of a ${p^j}^{th}$ root of a unit in $B$ that does not reduce to a $p^{th}$ power in the residue field of $A$. Similarly, we say that $B/A$ is \textit{potentially of $\mu_{p^j}$-type} if the base-change by some field, $K^{\prime}$, is of $\mu_{p^j}$-type. The former is equivalent to $B/A$ having the structure of a $\mu_{p^j}$-torsor.

Now we assume that $G$ is of the form $\mathbb{Z}/p^s \rtimes \mathbb{Z}/m$ and is branched at the points $P_1, \ldots, P_r$, $r>1$. As we saw in \S 2, we can decompose $f$ into an \'etale $\mathbb{Z}/p^s$-cover $Y \to Z$ and a branched $\mathbb{Z}/m$-cover, $Z \to E$.

Let $\chi: \mathbb{Z}/m \to \mu_m(K)$ be a character of order $m$. We have the reduction $\bar{\chi}$ of $\chi$, which gives a map $\mathbb{Z} \to \mathbb{F}_p^{\times}$. We then may use $\bar{\chi}$ to obtain a semidirect product, $\mathbb{Z}/p \rtimes \mathbb{Z}/m$. The $p$-cyclic cover above $Z$ is \'etale, so corresponds to a nontrivial class in $H^1_{et}(Z, \mathbb{Z}/p)_{\chi}$. Letting $J_Z$ denote the Jacobian of $Z$, we have a canonical isomorphism $H^1_{et}(Z, \mathbb{Z}/p)_{\chi} \cong J_Z[p]_{\chi}(-1)$. Moreover, we have that $J_Z[p]_{\chi}(-1) = \text{Hom}_{\mathbb{F}_p}(\mu_p(\bar(K)), J_Z[p]_{\chi})$. Choosing a $p^{th}$ root of unity in a fixed algebraic closure $\bar{K}$ of $K$, we identify $J_Z[p]_{\chi}(-1)$ with $J_Z[p]_{\chi}$. We have the following result concerning the structure of the Jacobian, $\text{Jac}(\bar{Z})[p]_{\bar{\chi}}$, of the special fiber $\bar{Z}$, which we denote $\bar{J}[p]_{\bar{\chi}}$:

\begin{prop} 
\label{jacobian}

Let $\bar{g}: \bar{Z} \to \bar{E}$ be an $m$-cyclic cover given, birationally, by the equation $z^m = f_r \cdot f_u$, where $f_r$ and $f_u$ are rational functions on $E$ corresponding to, respectively, a divisor in the form $\sum_{i=1}^r a_i P_i$ and an $m$-divisible divisor on $E$, $\sum_{j=1}^{\rho} m b_j Q_j$ with branch locus $\left\{P_1, \ldots, P_r \right\}$, $r \geq 2$, satisfying $\sum a_i = m$. Then $\bar{J}[p]_{\bar{\chi}} \cong \big( \mathbb{Z}/p \big)^{r-1} \times \mu_p$.

\end{prop}

\begin{proof} 

The action of $\mathbb{Z}/m$ on $\bar{g}_*\mathcal{O}_{\bar{Z}}$ arising from that on $\bar{Z}$ gives us the decomposition
\[ \bar{g}_*\mathcal{O}_{\bar{Z}} = \bigoplus \mathcal{L}_{\bar{\psi}}, \] 
where $\bar{\psi}$ runs through all characters $\mathbb{Z} \to k^{\times}$, into isotypical line bundles.

The rational function $z$ will have zeros of order $b_j$ at each $Q_j$ and a pole of order $\frac{1}{m} \big( \sum a_i + m \sum b_j \big) = 1+\sum b_j$ at $\infty$. This function is a rational section of $\bar{\chi}$, so we get $\text{deg}(\mathcal{L}_{\bar{\chi}}) = \sum b_j - (1+\sum b_j) = -1$. Similarly, $z^{m-1}$ has zeros of order $a_i - 1$ at each $\bar{x}_i$, zeros of order $b_j(m-1)$ at each $P_j$ and a pole of order $(1+\sum b_j)(m-1)$ at $\infty$, so that $\text{deg}(\mathcal{L}_{\bar{\chi}^{-1}}) = (m-r) + (m-1)\sum b_j - (1+\sum b_j)(m-1) = 1-r$. So $H^1(\bar{Z}, \mathcal{O}_{\bar{Z}})_{\bar{\chi}}$ and $H^1(\bar{Z}, \mathcal{O}_{\bar{Z}})_{\bar{\chi}^{-1}}$ have respective dimensions 1 and $r-1$.

Combining \cite{oda}, Corollary 5.11, and \cite{wed}, Theorem 2.1, we may compute the rank of $\bar{J}[p]_{\bar{\chi}}$ as a group scheme as 
\[\dim_k(H^1_{\text{dR}}(\bar{Z})_{\bar{\chi}}) = \dim_k(H^1(\bar{Z}, \mathcal{O}_{\bar{Z}})_{\bar{\chi}}) + \dim_k(H^0(\bar{Z}, \Omega_{\bar{Z}/k})_{\bar{\chi}}) = r\]
since $H^0(\bar{Z}, \Omega_{\bar{Z}/k})_{\bar{\chi}}$ is dual to $H^1(\bar{Z}, \mathcal{O}_{\bar{Z}})_{\bar{\chi}^{-1}}$.

There is a canonical $k$-linear isomorphism $\text{Lie}(J[p]) \cong H^1(\bar{Z}, \mathcal{O}_{\bar{Z}})$; it is compatible with the $\mathbb{Z}/m$ action, so we also have an isomorphism of $\bar{\chi}$-eigenspaces. So $\text{Lie}(J[p]_{\bar{\chi}})$ has dimension $1$ and $J[p]_{\bar{\chi}}$ is isomorphic to $\big(\mathbb{Z}/p \big)^{r-1} \times G$, where $G$ is either $\mu_p$ or $\alpha_p$. Since $\bar{J}[p]_{\bar{\chi}^{-1}}$ is dual to $\bar{J}[p]_{\bar{\chi}}$, we also have that $\bar{J}[p]_{\bar{\chi}^{-1}}$ is isomorphic to $\big(\mu_p \big)^{r-1} \times G^{\prime}$, where $G^{\prime}$ will be dual to $G$. Since $\text{Lie}(J[p])_{\bar{\chi}^{-1}}$ has dimension $r-1$, we must have $G^{\prime} = \mathbb{Z}/p$, and so $\bar{J}[p]_{\bar{\chi}} = \big( \mathbb{Z}/p \big)^{r-1} \times \mu_p$, as we wanted.

\end{proof}

We say that $f: Y \to E$ is of \textit{multiplicative type} when $\sum_{i=1}^r a_i = m$, as is the case in the previous proposition. The following is analogous to \cite{obus17}, Proposition 3.7 and gives the reasoning behind the terminology \textit{multiplicative type}:

\begin{prop} 
\label{multtype}

Let $f: Y \to X = E$ be a $G = \mathbb{Z}/p^s \rtimes \mathbb{Z}/m$-Galois cover of multiplicative type. Then any semistable model $f^{ss}: Y^{ss} \to X^{ss}$ of $f$ has multiplicative reduction above every inseparable component of $\bar{X}^{ss}$.

\end{prop} 

\begin{proof}

If $s=1$, then by Proposition \ref{jacobian} and \cite{raynaud90}, p. 190, the $p$-cyclic cover $Y \to Z$ reduces to a $\mu_p$-cover above the generic point of each component, and so will have multiplicative reduction above every component.

If $s>1$, then $f$ has a quotient $\mathbb{Z}/p \rtimes \mathbb{Z}/m$-cover, $h$, which will also be of multiplicative type, so that any semistable model of $h$ will have multiplicative reduction over every component of $\bar{X}^{ss}$. In particular, every component of $\bar{X}^{ss}$ will be a $p^s$-component for $f$. Let $K^{ss}$ be a field of definition for $f^{ss}$, a semistable model for $f$; since $K$ has algebraically closed residue field, \cite{obus17}, Lemma 2.2 implies that $f^{ss}$ has multiplicative reduction over all of $\bar{X}$, as we wanted.

\end{proof}

\subsection{The auxiliary cover} Raynaud's proof of potentially good reduction in \cite{raynaud99} the case of a $G$-cover having a $\mathbb{Z}/p$ Sylow $p$-subgroup relies on the construction of the \textit{auxiliary cover}. Obus (\cite{obus10}) generalized this construction to the case of a \textit{cyclic} Sylow $p$-subgroup, whose properties we recall here (cf. \cite{raynaud99}, \cite{obus10}):

\begin{const}
\label{auxcover}

Over some finite extension $K^{\prime}$ of $K$, we construct the \textit{auxiliary cover} of $f$, a $G^{aux}$-Galois cover $f^{aux}: Y^{aux} \to X^{aux}$, with $G^{aux} \leq G$, having the following properties:

\begin{enumerate}

\item[(i)] $(X^{aux})^{ss} = X^{st}$ and $\bar{X}^{aux} = \bar{X}$.

\item[(ii)] There is an \'etale neighborhood $Z$ of the union of the inseparable components of $\bar{X}$ with the property that, as covers, $f^{st} \times_{X^{st}} Z \cong \big(Ind_{G^{aux}}^G (f^{aux})^{st} \big)\times_{X^{st}} Z$.

\item[(iii)] The branch locus of $f^{aux}$ consists of a branch point $x_b$ of index $m_b$ for each \'etale tail $\bar{X}_b$ of $\bar{X}$ for which $m_b > 1$. If $\bar{X}_b$ is primitive, so that it contains a branch point of $f$, $x_b$ is this corresponding branch point. If $\bar{X}_b$ is new, $x_b$ specializes to a smooth point.

\item[(iv)] Given an \'etale tail $\bar{X}_b$ of $\bar{X}$ and an irreducible component $\bar{V}_b$ of $\bar{Y}^{aux}$ above $\bar{X}_b$ with effective ramification invariant $\sigma_b^{aux}$, we have $\sigma_b^{aux} = \sigma_b$.

\item[(v)] If $N$ is the maximal prime-to-$p$ normal subgroup of $G^{aux}$, then $G^{aux}/N \cong \mathbb{Z}/p^s \rtimes \mathbb{Z}/m_{G^{aux}}$ with $s \geq 1$.

\end{enumerate}

\end{const}

We denote by $G^{st}$ the quotient in (v); we then call the $G^{str}$-quotient cover of $f^{aux}$ the \textit{strong auxiliary cover}, denoted $f^{str}: Y^{str} \to X^{str}$. The upper-numbered filtration for the  higher ramification groups is unaffected by taking quotients, so both the auxiliary cover and the strong auxiliary cover will have the same effective ramification invariants as the original cover. By passing to either cover, we allow ourselves the convenience of studying the ramification with a simplified Galois group at the expense of a potentially larger branch locus. 

Strong auxiliary covers arising from one-point covers have the following property:

\begin{prop}
\label{auxtype}

If the strong auxiliary cover of a one-point cover with prime-to-$p$ branching ramifies, it is of multiplicative type.

\end{prop}

\begin{proof}

Let $\bar{X}_b$ be an \'etale tail of $\bar{X}^{str}$ with effective ramification invariant $\sigma_b$. We have seen that passing to the strong auxiliary cover preserves effective ramification invariants, so by Proposition \ref{fractionpart} and Theorem \ref{vanishingcycles}, $\sum \langle \sigma_b \rangle = 1$. If there were only one \'etale tail, this would mean that $\sigma_b = 1$ and $f^{str}$ would be unramified, a contradiction. So we assume there are at least two \'etale tails. Then by Lemma \ref{bounds}, each $\sigma_b$ is a non-integer, so that each \'etale tail has $m_b > 1$. So, applying Proposition \ref{fractionpart}, we see that $\sum \langle \sigma_b \rangle = \sum \frac{a_i}{m} = 1$, so that $\sum a_1 = m$ and $f^{str}$ is of multiplicative type.

\end{proof}

\begin{corollary}
\label{multiplicativereduction}

Let $f: Y \to E$ be a branched one-point cover with bad reduction and prime-to-$p$ branching indices. Then the stable model of $f$ has multiplicative reduction over the original component.

\end{corollary}

\begin{proof}

By Propositions \ref{multtype} and \ref{auxtype}, any semistable model of the strong auxiliary cover, $f^{str}$, has multiplicative reduction over the original component. Since $f^{str}$ is a prime-to-$p$ quotient of the auxiliary cover $f^{aux}$, any semistable model of $f^{aux}$ will also have multiplicative reduction over the original component. The cover, $f$, will be isomorphic to the disjoint union of copies of $f^{aux}$ over an \'etale neighborhood of the original component, so that the stable model of $f$ will also have multiplicative reduction over the original component, as we wanted.

\end{proof}

\section{Potentially Good Reduction}

Now let $f: Y \to X$ be a cover as in \S 2 with the additional requirement that $X \cong E$, an elliptic curve having good reduction over $K$.

\begin{lemma}
\label{unbranched}

Let $f: Y \to E$ be an \'etale cover defined over $K$ of the elliptic curve $E$, where $E$ has good reduction over $K$. Then $f$ has good reduction over $K$.

\end{lemma}

\begin{proof}

By the Riemann-Hurwitz formula, $Y$ must also be an elliptic curve, so $f$ is an isogeny. Isogenous elliptic curves either both have good or bad reduction over a given field (see, for instance, \cite{silverman}, Corollary 7.2); since $E$ has good reduction, so must $Y$, so that $f$ has good reduction, as we wanted.

\end{proof}

\begin{lemma}
\label{abelian}

Let $f: X \to E$ be one-point cover with abelian Galois group $G$. Then $f$ is, indeed, an \'{e}tale cover $f: X \to E$.

\end{lemma}

\begin{proof}

The \'etale fundamental group $\pi_1(E \setminus \{O\})$ of $E \setminus \{O\}$ has the presentation $\{a,b,c \, | \, aba^{-1}b^{-1}c=1\}$. The map $\pi_1(E \setminus \{O\}) \to G$ factors via the abelianization $\pi_1(E \setminus \{O\}) \to \pi_1(E \setminus \{O\})^{ab} = \{ab \, | \, aba^{-1}b^{-1} = 1\} = \pi_1(E)$, so that $f$ corresponds to an element of $\pi_1(E)$. So the map $f$ is, indeed, \'etale, as we wanted.

\end{proof}

In particular, one-point covers having abelian Galois groups must have good reduction.

The following is analogous to \cite{obus17}, Lemma 4.5:

\begin{lemma}
\label{thingy}

Let $f: Y \to E$ be a ramified cover, branched at $r \geq 1$ points, with tame branching indices; let $\bar{V}$ be an irreducible component above the original component above the original component with decomposition group $D_{\bar{V}}$. Then $m_{D_{\bar{V}}} > 1$.

\end{lemma}

\begin{proof}

It suffices to work with the strong auxiliary cover $f^{str}$, which, in light of Lemma \ref{abelian}, we assume has more than one branch point. It will be enough to show that the decomposition group of a component above the original component is not a $p$-group.

The cover $f^{str}$ has a $\mathbb{Z}/m_{G_{str}}$-quotient cover given (as, for instance, in \S 2) birationally by $z^m = f_r \cdot f_u$, where the rational function $f_r$ corresponds to the divisor $\sum_1^{\rho} a_i P_i$ on $E$, with $0 < a_i < m$. It then suffices to show that the reduction of $f_r$ is not an $m^{th}$ power in $k(X)$.

The branch points of $f$ specialize to distinct points on the special fiber, so we have at least $r$ different residue classes among them. By Proposition \ref{fractionpart} and Theorem \ref{vanishingcycles}, $\sum_{i=1}^{\rho} a_i = m \cdot r$, so that some subset of the $a_i$ satisfy $0 < \sum a_{i} < m$. So the factor $f_r$ does not reduce to an $m^{th}$ power, as we wanted.

\end{proof}

The following is the key step in the proof of good reduction and is analogous to \cite{obus17}, Proposition 5.1. We recall the statement; the proof is identical in our situation, with Lemma \ref{thingy} assuming the role of \cite{obus17}, Lemma 4.5.

\begin{prop}
\label{badreduction}

Let $G$ be a finite group with nontrivial cyclic Sylow $p$-subgroup; let $k$ be an algebraically closed field of characteristic $p$; let $K_0$ = $\text{Frac}(W(k))$ and let $K$ be a finite extension with $e(K) < \frac{p-1}{m_G}$, where $e(K)$ is the absolute ramification index of $K$; let $f: Y \to E$ be a $G$-cover defined over $K$, where $E$ is an elliptic curve having good reduction over the valuation ring $R$ of $K$, branched at the distinct $K$-points $\{x_1, \ldots, x_r \}$, with $r \geq 1$; suppose that the branch points of $f$ specialize to distinct points on the special fiber of the good model of $E$ over $R$. If $f$ has bad reduction, then the stable model of $f$ does not have multiplicative reduction over the original component.

\end{prop}

We then obtain the following by applying this to the $r=1$ case.

\begin{theorem}

Let $G$ be a finite group with cyclic Sylow $p$-subgroup. Let $K$ be as in Proposition \ref{badreduction} and let $f: Y \to E$ be a one-point $G$-cover defined over $K$. Then $f$ has potentially good reduction.

\end{theorem}

\begin{proof}

If $f$ is unbranched, then $f$ has good reduction by Lemma \ref{unbranched}. Otherwise, \cite{raynaud99}, Lemme 4.2.13, implies that the branching indices of $f$ are tame. If $f$ were to have bad reduction, Corollary \ref{multiplicativereduction} would imply that $f$ has multiplicative reduction over the original component, contrary to Proposition \ref{badreduction}. So $f$ has potentially good reduction.

\end{proof}

\bibliographystyle{plain}
\bibliography{bibliography}

\begin{thebibliography}{10}

\bibitem{dm}
Pierre Deligne and David Mumford.
\newblock The irreducibility of the space of curves of given genus.
\newblock {\em Inst. Hautes \'Etudes Sci. Publ. Math.}, (36):75--109, 1969.

\bibitem{sga1}
Alexander Grothendieck.
\newblock {\em Rev\^etements \'etales et groupe fondamental (SGA 1)}, volume
  224 of {\em Lecture notes in mathematics}.
\newblock Springer-Verlag, 1971.

\bibitem{liu06}
Qing Liu.
\newblock Stable reduction of finite covers of curves.
\newblock {\em Compos. Math.}, 142(1):101--118, 2006.

\bibitem{obus12}
Andrew Obus.
\newblock Vanishing cycles and wild monodromy.
\newblock {\em Int. Math. Res. Not. IMRN}, (2):299--338, 2012.

\bibitem{obus10}
Andrew Obus.
\newblock Fields of moduli of three-point {$G$}-covers with cyclic
  {$p$}-{S}ylow, {II}.
\newblock {\em J. Th\'eor. Nombres Bordeaux}, 25(3):579--633, 2013.

\bibitem{obus17}
Andrew Obus.
\newblock Good reduction of three-point galois covers.
\newblock {\em Algebraic Geometry}, 4(2):247--26, 2017.

\bibitem{oda}
Tadao Oda.
\newblock The first de {R}ham cohomology group and {D}ieudonn\'e modules.
\newblock {\em Ann. Sci. \'Ecole Norm. Sup. (4)}, 2:63--135, 1969.

\bibitem{raynaud90}
Michel Raynaud.
\newblock {$p$}-groupes et r\'eduction semi-stable des courbes.
\newblock In {\em The {G}rothendieck {F}estschrift, {V}ol.\ {III}}, volume~88
  of {\em Progr. Math.}, pages 179--197. Birkh\"auser Boston, Boston, MA, 1990.

\bibitem{raynaud99}
Michel Raynaud.
\newblock Sp\'ecialisation des rev\^etements en caract\'eristique {$p>0$}.
\newblock {\em Ann. Sci. \'Ecole Norm. Sup. (4)}, 32(1):87--126, 1999.

\bibitem{serre}
Jean-Pierre Serre.
\newblock {\em Local Fields}, volume~67 of {\em Graduate Texts in Mathematics}.
\newblock Springer-Verlag, New York-Berlin, 1979.

\bibitem{silverman}
Joseph~H. Silverman.
\newblock {\em The Arithmetic of Elliptic Curves}, volume 106 of {\em Graduate
  Texts in Mathematics}.
\newblock Springer, Dordrecht, second edition, 2009.

\bibitem{wed}
Torsten Wedhorn.
\newblock De {R}ham cohomology of varieties over fields of positive
  characteristic.
\newblock In {\em Higher-dimensional geometry over finite fields}, volume~16 of
  {\em NATO Sci. Peace Secur. Ser. D Inf. Commun. Secur.}, pages 269--314. IOS,
  Amsterdam, 2008.

\bibitem{wewers03}
Stefan Wewers.
\newblock Reduction and lifting of special metacyclic covers.
\newblock {\em Ann. Sci. \'Ecole Norm. Sup. (4)}, 36(1):113--138, 2003.

\end{thebibliography}

\end{document}